\documentclass[12pt]{amsart}
\usepackage[english]{babel}
\usepackage[utf8]{inputenc} 
\usepackage[T1]{fontenc}
\usepackage{verbatim}
\usepackage{palatino}
\usepackage{amsmath}
\usepackage{mathabx}
\usepackage{amsbsy}
\usepackage{amssymb}
\usepackage{amsthm}
\usepackage{amsfonts}
\usepackage{graphicx}
\usepackage{esint}
\usepackage{mathtools}
\usepackage{overpic}
\usepackage{mathrsfs}
\usepackage{tabularx}
\usepackage{datetime}
\usepackage{enumerate}
\usepackage{xcolor}
\usepackage{tikz}
\usepackage[pagebackref]{hyperref}
\hypersetup{
   colorlinks,
    linkcolor={red!60!black},
    citecolor={blue!60!black},
    urlcolor={blue!90!black}
}

\renewcommand{\today}{\the\day\ \shortmonthname[\month] \the\year}
\thanks{W.O.R is supported in part by an NSERC Alliance International grant administered by Pablo Shmerkin and Joshua Zahl. P.S. is supported by NSERC through a Discovery Grant and an Alliance International Grant, and also by a Simons Fellowship. H.W. is supported by NSF CAREER DMS-2238818 and NSF DMS-2055544.}

\numberwithin{equation}{section}

\theoremstyle{plain}
\newtheorem{theorem}{Theorem}[section]

\newtheorem{proposition}[theorem]{Proposition}

\newtheorem{lemma}[theorem]{Lemma}

\newtheorem{corollary}[theorem]{Corollary}
\newtheorem{prop}[theorem]{Proposition}

\theoremstyle{definition}

\mathtoolsset{showonlyrefs}

\newcommand{\ka}{\kappa}
\newcommand{\1}{\mathbf{1}}

\newcommand{\e}{\varepsilon}

\newcommand{\N}{\mathbb{N}}
\newcommand{\Z}{\mathbb{Z}}

\newcommand{\diam}{\operatorname{diam}}

\newcommand{\R}{\mathbb{R}}

\newcommand{\spt}{\operatorname{spt}}

\title{Simple proofs of discretised projection theorems}
\author{William O'Regan}
\author{Pablo Shmerkin}
\author{Hong Wang}
\begin{document}
\maketitle
\begin{abstract}
We give a simple, short and self-contained presentation of  Bourgain's discretised projection theorem from 2010, which is a fundamental tool in many recent breakthroughs in geometric measure theory, harmonic analysis, and homogeneous dynamics. Our main innovation is a short elementary argument that shows that a discretised subset of $\R$ satisfying a weak ``two-ends'' spacing condition is expanded by a polynomial to a set of positive Lebesgue measure.
\end{abstract}

\section{Introduction and statement of main results}

\subsection{Introduction}

Bourgain's discretised projection theorem \cite{Bourgain10} ranks among the most fundamental ``expansion'' estimates in analysis. 

In recent years, this theorem has been an essential tool in the resolution of some longstanding problems in geometric measure theory and harmonic analysis, including the radial projection problem \cite{OSW24, Ren23}, the Furstenberg set problem in the plane \cite{OrponenShmerkin23b, RenWang23}, and the Kakeya problem in $\mathbb{R}^3$ \cite{WangZahl25a, WangZahl25b, WangZahl25c}. To be more precise, none of these works invoke Bourgain's theorem directly, but rather build on the improved Furstenberg estimate in \cite{OrponenShmerkin23}, which in turn uses Bourgain's theorem as a black-boxed key ingredient. Bourgain's projection theorem is also a key (indirect) component in recent progress on the Falconer distance set problem and related problems \cite{RazZahl24, ShmerkinWang25}.

In a different direction, Bourgain's projection theorem has been a key tool in obtaining quantitative equidistribution estimates in ergodic theory. A striking early application (and partly a motivation for its development) was to equidistribution of orbits of nonabelian semigroups on the torus \cite{BFLM11}. Bourgain's theorem is also a crucial component in the recent breakthrough on exponential equidistribution of random walks on quotients of $\textrm{SO}(2,1)$ and $\textrm{SO}(3,1)$ \cite{BenardHe24}, and the related work establishing Khintchine's theorem for self-similar measures \cite{BHZ24}. Again, these recent works do not invoke Bourgain's theorem directly, but rather build on higher rank \cite{He20} and nonlinear \cite{Shmerkin23} versions of it, which in turn apply it as a black box.

Bourgain's original proof is quite involved and densely written. Given its fundamental nature and wide applicability, it is desirable to have a simpler, shorter proof and a more accessible presentation. Achieving this is the purpose of this note.

We comment on the related work of Guth, Katz and Zahl \cite{GKZ21}, who provided a simple and quantitative proof of Bourgain's discretised \emph{sum-product} theorem \cite{bou03}. While there is a clear similarity in spirit, deducing the projection theorem from the sum-product theorem is not straightforward. Our approach is different from that of \cite{GKZ21}, being based instead on the ideas of Edgar and Miller in solving the classical Erd\H{o}s-Volkmann ring conjecture \cite{EdgarMiller03}. It also has the advantage of working directly under weak ``two-ends'' spacing assumptions, as in Bourgain's original theorem. To be more precise, a modification of Edgar and Miller's approach is used to obtain the key expansion estimate, Theorem \ref{thm.expansion} below. The strength of this estimate also allows us to considerably simplify the proof of a key intermediate estimate of sum-product type, Theorem \ref{thm.ring}. To deduce Theorem \ref{thm.projhaus} from Theorem \ref{thm.ring}, we follow Bourgain's scheme, but with some additional simplifications.

Bourgain's original approach to both sum-product and projections involves first proving a structure theorem for sets with ``small'' sumsets, and then using this structure to derive expansion under multiplication. By contrast, our approach (as well as that of \cite{GKZ21}) is more direct, and involves addition and multiplication in an intertwined manner from the start.

We hope that this article will also be useful as an entry point to the recent literature on projection theory. One key feature of this area is the interplay between additive combinatorics and geometric measure theory. In Section \ref{sec:prelim} we have collected some key tools from these two areas that appear over and over in projection theory, and the way they are used here is typical of their use in the area.

\subsection{Main results}

Let $0 < \kappa \leq d,$ and let $C > 1, \delta >0$. We say that $A \subset \R^d$ is a \textit{$(\delta,\kappa,C)$-set} if it is a finite union of balls of radius $\delta,$ and $A$ satisfies the non-concentration condition
\begin{equation}
|A \cap B(x,r)| \leq Cr^\kappa|A| \text{ for all } x \in A \text{ and } r \geq \delta.
\end{equation}
Here, and throughout $|\cdot |$ denotes the Lebesgue measure on the ambient space.

We say that a Radon probability measure supported $\mu$ on $\R^d$ is a $(\kappa,C)$\textit{-measure} if $\mu$ satisfies the following non-concentration condition.
\begin{equation}
    \mu(B(x,r)) \leq Cr^\kappa \text{ for all } x\in \spt \mu \text{ and } r > 0.
\end{equation}

Given a set $A\subset\R$, we denote by $NA$ the $N$-fold sum-set, and by $A^{(N)}$ the $N$-fold product-set, see \S\ref{subsec:notation} for precise definitions. The below expansion estimate, with superficially weaker assumptions, and with a stronger conclusion, is \cite[Theorem 6]{Bourgain10}. It is the main building block for which the other results will follow. 
\begin{theorem}\label{thm.expansion}
	Let $0 < \kappa < 1$ and let $C >0$. Then there exists an integer $N$ depending on $C$ and $\kappa$ only so that the following holds. Let $\mu$ be a $(\kappa,C)$-measure supported on $[-C,C]$. Write $K := \spt \mu$. Then,
    \[ 
        |NK^{(N)} - NK^{(N)}| \gtrsim_{\kappa,C} 1.
    \]
\end{theorem}
Of note in the above result is that we get expansion all the way to positive Lebesgue measure, while in \cite{GKZ21} the conclusion is only that the set grows by a factor $\delta^{-c}$ for some $c>0$ (and only under a strong spacing assumption), and it is non-trivial to iterate this to get positive measure. Our proof of Theorem \ref{thm.expansion} is very short and relies only on Marstrand's projection theorem (Theorem \ref{thm.marstrand} below).

Below is Bourgain \cite[Theorem 2]{Bourgain10}, which is a hybrid between the discretised sum-product theorem and the discretised projection theorem. Using Theorem \ref{thm.expansion} as a starting point allows us to give a  proof that is substantially shorter and simpler than Bourgain's original one.
\begin{theorem}\label{thm.ring}
        Let $0 < \kappa \leq \sigma <1$. There exist $c, \e, \delta_0 \in (0,1]$, depending on $\kappa$ and $\sigma$ only so that the following holds for all $0 < \delta < \delta_0$. Let $A \subset \R$ be a $(\delta,\kappa,\delta^{-\e})$-set of measure $\le\delta^{1-\sigma}$. Let $\mu$ be a $(\kappa,\delta^{-\e})$-measure supported on $[0,1]$. There exists $x \in \spt \mu$ so that 
    \begin{equation}
        |A+xA| > \delta^{-c}|A|.
    \end{equation}
\end{theorem}
We emphasize that the above result only assumes a ``two-ends'' non-concentration condition on $A$, that is, it is a $(\delta,\kappa,\delta^{-\e})$-set, where $\kappa>0$ is arbitrarily small and, in particular, may be much smaller than $\sigma$. This is also the case in Bourgain's original formulation but not in \cite{GKZ21}. While there are known methods to reduce to this case from the strongest non-concentration (i.e. $\kappa=\sigma$), we believe that it is valuable to have a simple and direct proof that works under weak non-concentration assumptions.

Finally, below is Bourgain's result on projections, cf. \cite[Theorem 4]{Bourgain10}. In fact, we state a slightly streamlined version that is more suitable for direct application; it matches that of He \cite{He20} for the special case of projections from $\R^2$ to $\R$.

\begin{theorem}\label{thm.projhaus}
    Let $0 < \alpha <2, \beta, \kappa > 0$. There exist $\eta > 0$, $\e, \delta_0 \in (0,1]$, depending on $\alpha, \beta,$ and $\kappa$ only, so that the following holds for all $0 < \delta < \delta_0$. 
    
    Let $\mu$ be a $(\kappa,\delta^{-\e})$-measure on $S^1$. Let $E \subset B^2(0,\delta^{-\e})$ be a $(\delta,\beta,\delta^{-\e})$-set of measure $\le\delta^{2-\alpha}$.  There exists $\Theta \subset \spt\mu$ with $\mu(\Theta) > 1- \delta^\e$  so that for all $\theta \in \Theta$ we have
    \begin{equation}
        |\pi_\theta(G)|_{\delta} > \delta^{-\eta}|E|_{\delta}^{1/2} \text{ for all } G \subset E \text{ with } |G| > \delta^\e|E|.
    \end{equation}
\end{theorem}

\medskip

\subsection*{Acknowledgements}
W.O.R would like to thank Andr\'as M\'ath\'e for his numerous suggestions.

\section{Preliminaries}
\label{sec:prelim}

\subsection{Notation}
\label{subsec:notation}

For two functions $f:[0,1] \rightarrow [0,\infty)$ we write $f \lesssim g$ if there exists a constant $C > 0$ so that $f(x) \leq Cg(x)$ for all $x \in [0,1]$. The notation $\sim$ will mean both $\lesssim$ and $\gtrsim$. If we subscript $\lesssim$ with another value this will be used to emphasise dependence. Typically our functions will use the variable $\delta$ or $r$. 

By a measure we always mean a Borel probability measure. For a measure $\mu$ and a set $A$ denote $\mu_{|A}$ to be the measure \textit{conditioned} on $A,$ i.e $\mu_{|A}(\cdot) := \tfrac{\mu(\cdot \cap A)}{\mu(A)}$. 

 For a set $A \subset \R$ and for a lower-case letter $x \in \R$ we denote
\begin{equation}
    xA := \{xa: a \in A\}.
\end{equation}
For an upper-case letter $N \in \N$ we denote the \textit{$N$-fold sum-set} by
\begin{equation}
    NA := \{a_1 + \cdots + a_N : a_j \in A, 1 \leq j \leq N\}. 
\end{equation}
We also denote the \textit{$N$-fold product-set} by
\begin{equation}
    A^{(N)} := \{a_1\cdots a_N : a_j \in A, 1 \leq j \leq N\}. 
\end{equation}

Given a bounded set $A \subset \R^d,$ we denote its diameter by $\diam(A)$ and its $r$-neighbourhood by $A^{(r)}$.

Last, but not least, we let $|A|_{\delta}$ be the $\delta$-covering number of $A,$ i.e. the least number of balls of radius $\delta$ needed to cover $A$. Note that if $A$ is a union of balls of radius $\delta$, then $|A|_{\delta}\sim \delta^{-d}|A|$. In particular, for such sets many of the additive-combinatorial preliminaries can be stated in terms of Lebesgue measure rather than covering numbers.

\subsection{Additive combinatorics}

We recall some standard results from additive combinatorics and geometric measure theory that we will use, starting with the former. We emphasise that these results are standard and can be found in many references, and all the proofs are elementary and fairly short.

While the results below are typically stated for finite sets, it is a straightforward deduction to see that they also hold for $\delta$-covering numbers, by e.g. applying them to $\delta$-grid elements intersecting the relevant sets. See \cite{GKZ21} for details in some cases. This entails the loss of multiplicative constants that are of no consequence for our purposes.

We start with the Ruzsa triangle inequality; it is a direct consequence of the identity $x-y = (x-z) - (y-z)$, see \cite[Lemma 2.6]{TaoVu10}.
\begin{lemma}[Triangle inequality]\label{lem.triangleineq}
  Let $X,Y,Z$ be subsets of $\mathbb{R}$. Then
  \[
    |X-Z|_{\delta} \lesssim \frac{|X-Y|_{\delta} \, |Y-Z|_{\delta}}{|Y|_{\delta}}.
  \]
\end{lemma}

We will use the following version of the Pl\"{u}nnecke--Ruzsa inequality, see \cite[Theorem 6.1]{Ruzsa09}. 
\begin{prop}[Pl\"unnecke--Ruzsa inequality]\label{prop.pr}
  Let $X,Y_1,\ldots,Y_k$ be subsets of $\R$ with 
  \[
    |X+Y_i|_{\delta}\le \alpha_i |X|_{\delta}.
 \] 
 Then 
  \[
    |Y_1+\ldots+ Y_k|_{\delta} \lesssim_k \alpha_1\cdots \alpha_k |X|_{\delta}.
  \]
\end{prop}

Below is a simple corollary of Pl\"{u}nnecke--Ruzsa inequality and the triangle inequality. 
\begin{corollary}\label{cor.simple}
	Let $X, Y$ be subsets of $\mathbb{R}$, then 
	\begin{equation}\label{eq.simple}
		\max\{|X-X|_{\delta},|X+X|_{\delta}\}\lesssim \frac{ 	|X\pm Y|_{\delta}^2}{ 	|Y|_{\delta}}.
	\end{equation}
\end{corollary}
\begin{proof}
    It is enough to consider the case $|X+Y|_{\delta}$, since we can apply this case to $-Y$. The bound for $|X-X|_{\delta}$ follows from the triangle inequality with $Z=X$, and the one for $|X+X|_{\delta}$ follows from Pl\"{u}nnecke--Ruzsa inequality. 
\end{proof}

As a further (well-known) corollary, we have the following relationship between the sizes of $X+Y$ and $X-Y$.
\begin{corollary}
\label{cor.sum-to-difference}
  Let $X,Y$ be subsets of $\mathbb{R}$. Then
    \[
        |X-Y|_{\delta} \lesssim \frac{|X+Y|_{\delta}^3}{|Y|_{\delta}^2}.
    \]
\end{corollary}
\begin{proof}
    This follows from the triangle inequality with $Z=-Y$ followed by Corollary \ref{cor.simple}.
\end{proof}

We conclude with the Balog-Szemerédi-Gowers theorem in the form we will use it, see \cite[Theorem 2.29 and Ex (6.4.10) on p. 267]{TaoVu10}.

\begin{proposition}[Balog-Szemerédi-Gowers]\label{prop.bsg} There exists a constant $C> 0$ so that the following holds. Let $A,B$ be bounded subsets of $\R$ and let $G \subset A \times B$ be  such that for some $K >1$ we have
\begin{equation}
    |G|_{\delta} \geq |A|_{\delta}|B|_{\delta}/K,
\end{equation}
and \begin{equation}
    |A\overset{G}{+}B|_{\delta} := |\{x+y: (x,y) \in G\}|_{\delta} \leq K|A|_{\delta}^{1/2}|B|_{\delta}^{1/2}. 
\end{equation}
    Then there exist $A' \subset A$ and $B' \subset B$ satisfying
    \begin{equation}
        |A'|_{\delta} \gtrsim K^{-C}|A|_{\delta},
    \end{equation}
    \begin{equation}
        |B'|_{\delta} \gtrsim K^{-C}|B|_{\delta},
    \end{equation}
    \begin{equation}
        |A'+B'|_{\delta} \lesssim K^{C}|A|_{\delta}^{1/2}|B|_{\delta}^{1/2},
    \end{equation}
    and 
    \begin{equation}
        |G \cap (A' \times B')|_{\delta} \gtrsim K^{-C}|A|_{\delta}|B|_{\delta}.
    \end{equation}
\end{proposition}

\subsection{Geometric measure theory}

We now recall some standard definitions and results from geometric measure theory. For a measure $\mu$ we let
\begin{equation}
    I_s(\mu) = \iint \frac{1}{|x-y|^s}\,d\mu(x)d\mu(y)
\end{equation}
be the \textit{$s$-energy} of $\mu$. 

The following standard lemmas show that there is a close connection between $s$-energy and $(\delta,s,C)$-sets. For a $(\delta,s,C)$-set $A \subset \R^d$, let $\mu_A$ denote the uniform probability measure on $A$. A straightfoward calculation yields the following. 
\begin{lemma}\label{lem.energy}
    Let $A$ be a $(\delta,s,C)$-set. Then
     \begin{equation}
        I_t(\mu_A) \lesssim_{t,s} C \quad \text{ for } t < s
     \end{equation}
\end{lemma}
\begin{proof}
    Fix $x$ and a dyadic scale $2^{-j}$. We have
    \[
        \int_{2^{-j}\le |x-y|< 2^{1-j}} \frac{1}{|x-y|^t} \,d\mu_A(y) \lesssim 2^{jt} \mu_A(B(x,2^{-j})) \lesssim C 2^{j(t-s)}.
    \]
    Integrating over $x$ and then adding up over $j$ yields the claim. (If $2^{-j}<\delta$ we get an even better bound since $A$ is a union of balls of radius $\delta$.)
\end{proof}

\begin{lemma} \label{lem.energy-to-frostman}
    Given a scale $\delta>0$, $s>0$ and constants $K,L\ge 1$, the following holds. 
    
    Let $\mu$ be a measure with $I_s(\mu) \leq K$. Then  there is a $(\delta,s/2,O_d(KL))$-set $A \subset \spt\mu$ with $\mu(\R^d\setminus A)\lesssim_s 1/L$.
\end{lemma}
\begin{proof}
    Fix $\delta$, which we may assume is of the form $2^{-n}$ for some $n\in\mathbb{N}$. For each $0\le j<n$, partition $\R^d$ into dyadic cubes $Q_{j,i}$ of side length $2^{-j}$, and note that 
    \[
         2^{js}\sum_i \mu(Q_{j,i})^2 \lesssim  I_s(\mu) \leq  K.
    \]
    Hence,
    \[
           \mu(E_j) \lesssim 2^{-js/2}/L, \quad\text{where } E_j = \bigcup \bigl\{ Q_{j,i} : \mu(Q_{j,i}) \geq K L 2^{-js/2} \bigr\}.
    \]
    The claim follows with $A = \spt\mu \setminus \bigcup_{j=0}^{n-1} E_j$.
\end{proof}

Next, we recall a variant of Marstrand's projection theorem, see \cite[Theorem 4.3]{Mattila15}.
\begin{theorem}[Marstrand] \label{thm.marstrand}
     Let $\mu$ be a measure on $\R^d$ with $I_1(\mu) < \infty$ and let $A=\spt\mu$. Then, denoting Lebesgue measure on $S^{d-1}$ by $\sigma^{d-1}$,
    \[
        \int |\pi_\theta A| \,d\sigma^{d-1}(\theta) \lesssim_d  \frac{1}{I_1(\mu)}.
    \]
\end{theorem}

Finally, we will require (the proof of) Kaufman's projection theorem, see \cite[Theorem 3.2]{Shmerkin23} for this precise formulation.
\begin{theorem}[Kaufman]\label{thm.kauf}
    Let $0 < \kappa< \alpha \leq 1$. Let $\mu$ be a measure on $\R^2$ and let $\nu$ be an $(\alpha,C)$-measure on $S^1$.
    Then,
\begin{equation}
    \int I_{\kappa}(\pi_\theta\mu)d\nu(\theta) \lesssim_{\kappa,\alpha} C  I_{\kappa}(\mu).
\end{equation}
\end{theorem}

\section{Proof of Theorem \ref{thm.expansion}}

Theorem \ref{thm.expansion} is false over $\mathbb{C}$, since in that case the set $A$ could be a subset of $\R$. In our proof of Theorem \ref{thm.expansion} what distinguishes $\R$ from $\mathbb{C}$ is the fact that the lattice $\Z^n$ has full rank in $\R^n$ (but not in $\mathbb{C}^n$). This is used via Blichfeldt's principle, which we state and prove in the following lemma for completeness.
\begin{lemma} \label{lem.latticepoints}
    Let $V\subset\R^n$ be measurable and let $L:\R^n\to\R^n$ be an invertible linear map. Then there is a translation $x$ such that
    \[
        \#\left\{V \cap \bigl(L\Z^n +x\bigr) \right\} 
        \ge   \det(L)^{-1} |V|.
    \]
\end{lemma}
\begin{proof}
   Writing $x=Ly$, it is enough to show that there is $y$ such that
    \[
        \#\left\{L^{-1}V \cap (\Z^n +y) \right\} 
        \ge |V|.
    \]
    For any $y\in\R^n$,
    \[
        \#\left\{L^{-1}V \cap (\Z^n +y) \right\} 
        = \sum_{k\in \Z^n} \1_{L^{-1}V-k}(y).
    \]
    Integrating over $y\in [0,1]^n$ and interchanging the sum and the integral, we get
    \[
        \int_{[0,1]^n} \#\left\{L^{-1}V \cap (\Z^n +y) \right\} \,dy
        = \sum_{k\in \Z^n}  \bigl|[0,1]^n \cap (k - L^{-1}V)\bigr| = |L^{-1}V|.
    \]
    Hence, there exists $y\in [0,1]^n$ such that 
    \[ 
        \#\left\{L^{-1}V \cap (\Z^n +y) \right\} \ge |L^{-1}V|= \det(L)^{-1} |V|,
    \] as desired.
\end{proof}

The following lemma is inspired by the method of \cite{EdgarMiller03}, but it introduces the main new idea in this paper: in order to force non-injectivity of certain linear maps, we place many translated copies of $A^n$ inside a slab, with the difference between any two translations in $N_0 A- N_0 A$ for some $N_0$. See Figure \ref{fig.slab} for an illustration.
\begin{lemma} \label{lem.expansion}
Let $A\subset\R$ be a compact set. Fix $v\in [1/2,1]^n$ (for some $n\ge 2$) and write $\pi(x)=v\cdot x$, $x\in\R^n$. Suppose $|\pi(A^n)|>0$.

Then there exists a positive integer $N$, which depends on $n,\diam(A),|\pi(A^n)|$ only, and can be taken uniform on compact subsets of $\N\times (0,\infty)\times (0,\infty)$, so that for $A_1 := NA^{(2)}- NA^{(2)}$, and after possible reordering of the $v_j$, we have
\[ 
    |v_1 A_1 + \cdots + v_{n-1}A_1| \ge \diam(A)\cdot\left|\pi(A^n)\right|.
\]
\begin{proof}

    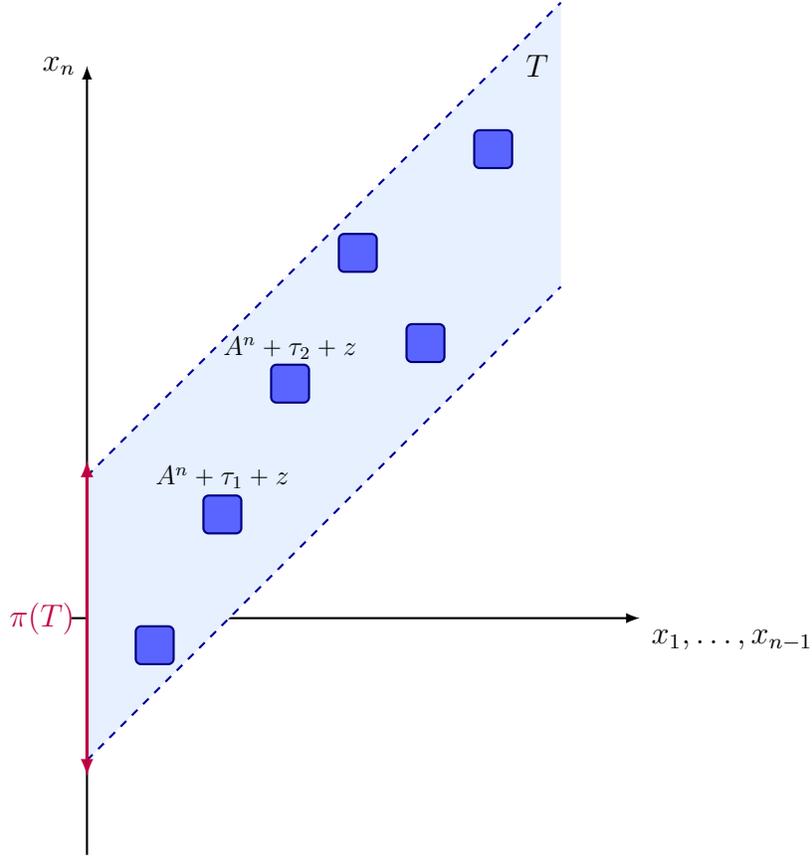
\begin{figure}[h!] \label{fig.slab}
    \centering
    \begin{tikzpicture}[scale=1.05,>=latex]

    % ---------- visual parameters ----------
    \def\xmin{0}
    \def\xmax{6}
    \def\a{1.8}         % half-width of slab between y = x ± a
    \def\boxr{0.24}     % half side-length of each small box (axis-aligned)
    \def\margin{0.30}   % safety margin from slab boundaries
    % we choose centers so that min center distance > 1 (>= sigma since sigma ≤ 1 in the lemma)

    % ---------- colors ----------
    \definecolor{slabfill}{RGB}{230,240,255}
    \definecolor{boxfill}{RGB}{90,100,255}

    % ---------- axes ----------
    \draw[->,thick] (-0.2,0) -- (7.0,0) node[below right] {$x_1,\dots,x_{n-1}$};
    \draw[->,thick] (0,-3.0) -- (0,7.0) node[left] {$x_n$};

    % ---------- oblique slab T = pi^{-1}([-100\Lambda,100\Lambda]) ----------
    % slab between y = x - a and y = x + a over [xmin, xmax]
    \fill[slabfill]
        (\xmin,\xmin-\a) -- (\xmax,\xmax-\a) -- (\xmax,\xmax+\a) -- (\xmin,\xmin+\a) -- cycle;
    \draw[thick,blue!60!black,dashed] (\xmin,\xmin-\a) -- (\xmax,\xmax-\a);
    \draw[thick,blue!60!black,dashed] (\xmin,\xmin+\a) -- (\xmax,\xmax+\a);

    \node at (5.7,7) {$T$};

     \draw[<->, very thick,purple] (0,-2) -- node[left] {$\pi(T)$} (0,2);

    % ---------- evenly spread, sigma-separated translates of A^n ----------
    % equal-spaced x positions; offsets s_j distributed across (-a+margin, a-margin)
    \pgfmathsetmacro{\w}{(\xmax-\xmin)/7} % 6 boxes at j=1..6 -> 7 segments
    % offsets chosen to cover the band fairly evenly and ensure separation
    % (each center is at (x_j, x_j + s_j))
    \foreach \j/\sj in {1/-1.20, 2/-0.40, 3/0.40, 4/1.20, 5/-0.80, 6/0.80}{
    \pgfmathsetmacro{\xc}{\xmin + \j*\w}
    \pgfmathsetmacro{\yc}{\xc + \sj}
     % safety: keep strictly inside the band by clamping (visual only; values already within)
    \pgfmathsetmacro{\ymin}{\xc - \a + \margin}
    \pgfmathsetmacro{\ymax}{\xc + \a - \margin}
    \pgfmathsetmacro{\ycc}{min(max(\yc,\ymin),\ymax)}
    % draw the box centered at (\xc,\ycc)
    \filldraw[fill=boxfill,draw=blue!50!black,thick,rounded corners=2pt]
        (\xc-\boxr,\ycc-\boxr) rectangle (\xc+\boxr,\ycc+\boxr);
    }

    % ---------- show explicit separation between two   neighboring boxes ----------
    % pick j=2 and j=3 for the arrow
    \pgfmathsetmacro{\xA}{\xmin + 2*\w}
    \pgfmathsetmacro{\yA}{\xA - 0.40}
    \pgfmathsetmacro{\xB}{\xmin + 3*\w}
    \pgfmathsetmacro{\yB}{\xB + 0.40}

    % optional minimal labels (avoid clutter)
    \node[scale=0.85] at (\xmin+2*\w,\xmin+1.8*\w-0.40+0.64) {$A^n+\tau_1+z$};
    \node[scale=0.85] at (\xmin+3*\w,\xmin+1.85*\w+1.20+0.64) {$A^n+\tau_2+z$};

    \end{tikzpicture}
    \caption{Slab $T$ of width and radius $\sim 1$ containing  mutually separated translates of $A^n$, and such that $\pi(T)$ is bounded. Once the number of translates is large enough (but still $\sim 1$), the projection $\pi$ can no longer be injective on their union. This allows us to eliminate one of the $v_j$ from the sumset, at the cost of replacing $A$ by $N A^{(2)}- N A^{(2)}$.}

    \end{figure}

    Write $\ker\pi= U e_n^\perp$, where $U$ is an orthogonal matrix and $e_n = (0,\ldots,0,1)$. For a parameter $R\ge 1$ to be chosen later, consider the set 
    \[
        V = U\bigl( B^{n-1}(0,R) \times [-1,1]\bigr).
    \]
    Note that $|V| = 2 \omega_{n-1} R^{n-1}$ (where $\omega_{n-1}$ is the volume of the unit ball in $\R^{n-1}$). Also, since $\|v\|\le \sqrt{n}$,
    \begin{equation} \label{eq.pi-V-small}
         \pi(V) \subset [-\sqrt{n},\sqrt{n}].
    \end{equation}
    Thanks to Lemma \ref{lem.latticepoints}, there exists a translation $z$ such that
    \[
        \#\{V\cap (2\diam(A)\Z^n+z)\} \ge [2\diam(A)]^{-n} |V| = 2^{1-n}\omega_{n-1} \diam(A)^{-n}R^{n-1}.
    \]
    We now choose $R\ge 1$ so that
    \[
        2^{1-n}\omega_{n-1} \diam(A)^{-n}R^{n-1} \ge 2[n \diam(A)+2\sqrt{n}]\lambda^{-1},
    \]
    where $\lambda=|\pi(A^n)|>0$. Applying Lemma \ref{lem.latticepoints}, we find $\tau_1,\ldots,\tau_M\in\Z^n$ and $z\in\R^n$ such that
    \begin{equation} \label{eq.M-lower}
        M\ge  2[n \diam(A)+2\sqrt{n}]\lambda^{-1},
    \end{equation}
    and the sets $A^n + 2 \diam(A)\tau_j + z$, $j=1,\ldots,M$, are contained in $V+A^n$. In particular,  by \eqref{eq.pi-V-small}, 
    \[
        \pi\bigl(A^n + 2 \diam(A)\tau_j + z\bigr) \subset [-\sqrt{n},\sqrt{n}] + \pi(A^n).
    \]
    Note that $\pi(A^n)$ is contained in an interval of length $n\diam(A)$. Hence,
    \[
        \left| \bigcup_{j=1}^M\pi \bigl(A^n + 2 \diam(A)\tau_j + z\bigr)\right| \le n\diam(A) + 2\sqrt{n} .
    \]
    On the other hand, since each set $\pi\bigl(A^n + 2 \diam(A)\tau_j + z\bigr)$ has measure $\lambda$, we have
    \[
        \sum_{j=1}^M \left|\pi\bigl(A^n + 2 \diam(A)\tau_j + z\bigr)\right| \ge M\cdot \lambda.
    \]
    Recalling \eqref{eq.M-lower}, by the pigeonhole principle there exist $i\neq j$ such that, setting $\ell=\tau_i-\tau_j\in\Z^n\setminus\{0\}$, we have $2\diam(A)\ell \in V-V$, and 
    \[
          \pi(A^n) \cap \pi\bigl(A^n + 2 \diam(A)\ell\bigr) \neq \emptyset.
    \]
    Thus, there are $x,y\in A^n$ such that
    \[
        \pi(x) = \pi(z), \quad \text{where } z= y+ 2 \diam(A)\ell.
    \]
    After rearranging, we may assume that $\ell_n\neq 0$. In particular, this means
    \[
      |z_n - x_n | \ge \diam(A).
    \]
    From $\pi(x) = \pi(z)$ we obtain
    \begin{equation}\label{eq.notinj3}
        v_n(z_n - x_n) =  \sum_{i=1}^{n-1} v_{i}(x_i - z_i).
    \end{equation}
    Since $|z_n - x_n| \geq \diam(A)$ we have
    \[
        \bigl|v_1(z_n-x_n)A + v_2(z_n-x_n)A + \cdots +  v_n(z_n-x_n)A\bigr| \ge \diam(A) \lambda.
    \]
    Substituting \eqref{eq.notinj3} in the above we have
    \begin{equation} \label{eq:sum-with-vn-eliminated}
        \left| \sum_{i=1}^{n-1} v_i(x_iA -z_i A + z_n A-x_n A)  \right| \ge \diam(A) \lambda.
    \end{equation}
    Now, since $z_i = y_i + 2 \diam(A)\ell_i$, and $\diam(A)\in A-A$ since $A$ is closed, we have 
    \[
    z_i \in A\pm 2|\ell_i| (A-A).
    \]
    Recall that $\ell\in (V-V)/(2\diam(A))$. Hence,
    \[
        |\ell_i| \le \frac{\diam(V)}{2\diam(A)} \le \frac{\sqrt{R^2+1}}{2\diam(A)}.
    \]
    Hence, each term in the sum in \eqref{eq:sum-with-vn-eliminated} is contained in $A_1 := N A^{(2)} - N A^{(2)}$, whenever
    \[
        N \ge  \frac{2\sqrt{R^2+1}}{\diam(A)} +2 .
    \]
    Since $R$ depends continuously on $n,\diam(A),\lambda$, the proof is complete.
\end{proof}

\begin{proof}[Proof of Theorem \ref{thm.expansion}]
All implicit constants in the proof are allowed to depend on $\kappa$ and $C$. %By replacing $\mu$ by its translate by $-\min(\spt\mu)$, we may assume that $0 \in \spt \mu$.

Let $K = \spt \mu$. Let $n$ be the least integer so that $n\kappa >1$, so that $n\sim 1$. Let $\nu=\mu^{\times n}$ which is a $(\kappa n, C^n)$-measure on $\R^n$. By Lemma \ref{lem.energy} we have $I_1(\nu) \lesssim 1$.

 By Theorem \ref{thm.marstrand}, we find $v\in [1/2,1]^n$ so that
    \begin{equation}
        |v_1 K + v_2K + \cdots +  v_{n}K| \gtrsim 1.
    \end{equation}
    The idea is now to apply Lemma \ref{lem.expansion} repeatedly to eliminate each of the $v_j$ in the sum above (at the cost of replacing $K$ by $NK^{(2)} - NK^{(2)}$ for some integer $N$ at each step) until only one term remains.

    Note that $\diam(K)\sim 1$ by the non-concentration property of $\mu$. Applying Lemma \ref{lem.expansion}, we get  $N_1 \sim 1$ so that 
    \begin{equation}
        |v_1 K_1 + v_2K_1 + \cdots + v_{n-1}K_1| \gtrsim 1,
    \end{equation}
    where $K_1 = N_1K^{(2)} - N_1K^{(2)}$. Applying the lemma again, $n-2$ times, each with different inputs $\diam(K_i)$, $\pi_j(K_i^j)$, but all $\sim 1$, we find $N_1,\cdots N_{n-1} \sim 1$ so that 
    \begin{equation}
        |NK^{(N)}-NK^{(N)}| \gtrsim 1,
    \end{equation}
    where (say) $N = (N_1\cdots N_{n-1})^{3^n}$.
\end{proof}
\end{lemma}

\section{Proof of Theorem \ref{thm.ring}}

We start with a variant of Theorem \ref{thm.ring} in which $\mu$ is $(\kappa, 1)$ instead of $(\kappa, \delta^{-\e})$. While superficially a much stronger condition, we will later reduce Theorem \ref{thm.ring} to this special case. We first apply Marstrand's Theorem to obtain expansion for $x_1A + \cdots + x_m A$, where $x_j \in N K^{(N)}- NK^{(N)}$ (which has positive measure by Theorem \ref{thm.expansion}). Then we use the triangle and Pl\"unnecke--Ruzsa inequalities many times to reduce the complexity of the sumset until we reach the desired form $xA + t^{-1}A$.

\begin{prop}\label{thm.ringstrong}
    Let $0 < \kappa \leq \sigma <1$. There exist $c, \e, \delta_0 \in (0,1]$, depending on $\kappa$ and $\sigma$ only, so that the following holds for all $0 < \delta < \delta_0$. Let $A \subset \R$ be a $(\delta,\kappa,\delta^{-\e})$-set of measure $\delta^{1-\sigma}$.  Let $\mu$ be a $(\kappa,1)$-measure supported on $[0,1]$. 

    Then  for any $t> \delta^{c}$,  there exists $x \in \spt \mu$ so that 
    \begin{equation}
        |t^{-1}A+xA| > \delta^{-c}|t^{-1} A|.
    \end{equation}
\end{prop}
%The lower bound $t> \delta^{c}$ is not essential, since when $t< \delta^{c}$, the problem becomes trivial. In our later application, $t$ has the form $\delta^{\eta}$ for some tiny $\eta\ll c$.
\begin{proof}
Note that $\mu[0,2^{-1/\kappa}] \leq 1/2$ by the $(\kappa,1)$-measure condition. Hence, replacing $\mu$ by its normalized restriction to $[2^{-1/\kappa},1]$, we may assume that $\mu$ is supported on points $x$ with $x\sim 1$, at the slight cost of changing the $(\kappa,1)$-measure condition to a $(\kappa,2)$-measure condition. %We may further assume that $1\in \spt\mu$; otherwise, letting $b=\max\spt\mu\in [0,1]$, we scale $\mu$ by $b^{-1}$ and replace $A$ by $bA$. Since $b\sim 1$, these changes are harmless. 

%By replacing $\mu$ by $\mu_{|[0,1]\setminus I},$ where $I = (\delta,(2C)^{-1/\ka}),$ we work with a $(\kappa,2C)$-measure. 

Set $K = \spt \mu$. Apply Theorem \ref{thm.expansion} to the (new) measure $\mu$, which tells us there exists an integer $N$, which depends on $\kappa$ only, so that $|X| \gtrsim 1$, where $X = NK^{(N)} - NK^{(N)}$.

Let $m$ be  the least integer so that $m\kappa > 1$. Let $\nu$ be the uniform probability measure on $A^m$. Then $\nu$ is an $(m\kappa, \delta^{-\e m})$-measure, and by Lemma \ref{lem.energy} we have
\[
 I_1(\nu) \lesssim  \delta^{-\e m}.
\]
Since $|X|\gtrsim 1$ and $X$ is contained in a ball of radius $\sim 1$, we can find a subset $W\subset X^m$ such that $|w|\sim 1$ for all $w\in W$ and $|W|\gtrsim 1$, which implies that $\sigma^{m-1}(\pi(W)) \gtrsim 1$, where $\pi: \R^m\setminus\{0\}\to S^{m-1}$ is the radial projection. We can then apply Marstrand's Theorem (Theorem \ref{thm.marstrand}) to obtain $x_1,\dots, x_{m} \in X$ so that
 \begin{equation}
     |x_1A + \cdots + x_m A| \gtrsim \delta^{\e m}. 
 \end{equation}
 Setting $k = mN$, there exist $y_1,\cdots, y_{2k} \in K^{(N)}$ so that
  \begin{equation}
     |y_1A + \cdots + y_kA - y_{k+1}A - \cdots - y_{2k}A| \gtrsim \delta^{\e m} \ge \delta^{-(1-\sigma)/2}|A| \ge \delta^{-(1-\sigma)/4}|t^{-1}A|, 
 \end{equation}
provided that $\e m < (1-\sigma)/2$ (recall that $m$ depends only on $\kappa$) and that $t \ge \delta^{(1-\sigma)/4}$.  It now follows from Proposition \ref{prop.pr} (Pl\"unnecke--Ruzsa inequality) that there are $y\in K^{(N)}$ and a sign $* \in \{+,-\}$ so that
\begin{equation}
    |t^{-1}A * y A| \gtrsim \delta^{-\tfrac{(1-\sigma)}{8k}}|t^{-1}A| .
\end{equation}
By Corollary \ref{cor.sum-to-difference}, in any case we have
\begin{equation} \label{eq.gain-with-product-element}
        |t^{-1}A - yA| \gtrsim \delta^{-\tfrac{(1-\sigma)}{24k}}|t^{-1}A|.
\end{equation}
Let $y = w_1\cdots w_N $ with $w_i\in K$, and recall that $w_i\sim 1$.

Write $z_j = w_1\cdots w_j$ for $1\leq j \leq N$ (with $z_0=1$). By  Lemma~\ref{lem.triangleineq} (the triangle inequality), for $1\le j \leq N$ we have
\begin{equation}
    |t^{-1}A - z_j A| \lesssim \frac{|t^{-1}A - z_{j-1} A| |z_{j-1}A - z_j A|}{|z_{j-1}A|} \sim \frac{|t^{-1}A - z_{j-1} A| |A-w_j A|}{|A|}.
\end{equation}
Multiplying out the above inequalities for $j=1,\cdots, N$, and recalling that $y=z_N$, we have
\begin{equation}
    |t^{-1}A - y A| \lesssim |t^{-1}A-A| \, \prod_{j=1}^N \frac{|A - w_j A|}{|A|}.
\end{equation}
By the triangle inequality again, for each $1\leq j \leq N$, we have
\begin{equation}
    |A - w_j A| \lesssim \frac{|t^{-1}A - A| |t^{-1}A - w_j A|}{|t^{-1}A|}.
\end{equation}
Plugging this into the previous inequality, and recalling that $t\ge \delta^c$ (where $c$ is yet to be determined), we get
\begin{equation}
    |t^{-1}A - y A| \lesssim |t^{-1}A - A|^{N+1}\cdot \left(\prod_{j=1}^N |t^{-1}A - w_j A| \right) \cdot \delta^{cN}|t^{-1}A|^{-2N}.
\end{equation}
Recalling \eqref{eq.gain-with-product-element}, we conclude that there is $w\in \{1,w_1,\ldots, w_N\}$  
\begin{equation}
    |t^{-1}A - w A| \gtrsim \delta^{\frac{cN}{2N+1}-\tfrac{1-\sigma}{24k(2N+1)}} |t^{-1}A|.
\end{equation}
If $w=1$, then by Corollary \ref{cor.simple} we have
\begin{equation}
    |t^{-1}A - w_1 A| \gtrsim |w_1 A|^{1/2} |t^{-1}A-A|^{1/2} \ge \delta^{cN/(4N+2) - (1-\sigma)/(48k(2N+1))} |t^{-1}A|.
\end{equation}
Another application of Corollary \ref{cor.sum-to-difference} yields the desired conclusion with 
\[ 
    c = \frac{1-\sigma}{400 k(2N+1)}.
\] 
\end{proof}

\begin{proof}[Proof of Theorem \ref{thm.ring}]
    We will use a well-known trick to reduce Theorem \ref{thm.ring} to Proposition \ref{thm.ringstrong}. Let $A$ and $\mu$ be as in the statement.

    For each closed interval $I$ of length $r(I)\in [\delta,1]$, let $m(I)=\mu(I)/r(I)^{\kappa/2}$. It is easy to see that $m$ achieves a maximum; let $I_0$ an interval where the maximum is attained, and set $r_0=r(I_0)$. Since $\mu$ is $(\kappa,\delta^{-\e})$, we
    have 
    \[
     r_0^{\kappa/2} \leq \mu(I_0) \leq r_0^{\kappa}\delta^{-\e},
    \]
    where the left-hand inequality follows from $m(I_0) \geq m([0,1]) = 1$. This implies that 
    \begin{equation} \label{eq.r0lower}
         \delta^{2\e/\kappa} \le r_0 \le 1.
    \end{equation}

    By the maximality of $m(I_0)$, we have $\mu(J)/\mu(I_0)\le (r/r_0)^{\kappa/2}$ for any interval $J\subset I_0$ of length $r\le r_0$. This shows that if $h$ renormalizes $I_0$ to $[0,1]$, then the measure $\nu = h_\#(\mu_{|I_0})/\mu(I_0)$ is a $(\kappa/2,1)$-measure. We can ensure $r_0> \delta^{c}$ by taking $\e$ small enough in terms of $c$ and $\kappa$, hence in terms of $\kappa$ and $\sigma$ only. We can therefore apply Proposition \ref{thm.ringstrong} to $\nu$ and $t=r_0^{-1}$ to conclude that there exist $c = c(\kappa,\sigma) > 0$ and $y \in \spt \nu = h(\spt \mu\cap I_0)$ so that
    \[
        |r_0^{-1}A + y A| > \delta^{-c}|r_0^{-1} A| \ge \delta^{-c}|A|.
    \]
    Let $x=h^{-1}(y)\in \spt\mu \cap I_0$. Then $y=r_0^{-1}(x - x_0)$, where $x_0$ is the left endpoint of $I_0$. Note that $x_0\in\spt\mu$, since otherwise we could have chosen a smaller interval to maximize $m$. 

    Putting this into the previous inequality and recalling \eqref{eq.r0lower}, we obtain
    \[
        |A+x A- x_0A| \ge |A + (x - x_0) A| > r_0 \delta^{-c}|A| > \delta^{-c/2}|A|,
    \]
    provided that $\e$ is small enough in terms of $\kappa,\sigma$. The Pl\"unnecke--Ruzsa inequality then yields that for some $y\in \{ 1, x, -x_0\}$ we have
    \[
        |A + y A| > \delta^{-c/6}|A|.
    \]
    If $y=x$ we are done. If $y=-x_0$, we get from Corollary \ref{cor.sum-to-difference} that $|A+x_0 A| > \delta^{-c/18}|A|$. If $y=1$, let $x_2=\max \spt\mu$. Since $\mu$ is a $(\kappa,\delta^{-\e})$-measure, we have $x_2\ge \delta^{\e/\kappa}$. Applying Corollary \ref{cor.simple}, we conclude 
    \[
        |A + x_2 A | \gtrsim |x_2 A|^{1/2} |A+A|^{1/2} \ge \delta^{\e/\kappa-c/12}|A| .
    \] 
    Taking $\e$ small enough in terms of $c$ and $\kappa$, we obtain the desired conclusion with $c/24$ in place of $c$.
\end{proof}

\section{Proof of Theorem \ref{thm.projhaus}}

To deduce Theorem \ref{thm.projhaus} from Theorem \ref{thm.ring}, we follow the outline of Bourgain's argument in \cite{Bourgain10}. Our main simplification is that we rely on Kaufman's projection theorem (Theorem \ref{thm.kauf}) to establish the non-concentration condition on the projections, instead of the more complicated ad hoc argument of \cite{Bourgain10}. 

Suppose the projection of $E$ is roughly $\alpha/2$ dimensional in $2$ well-separated directions $\theta,\theta'$. Thus, we can embed $E$ densely in a Cartesian product $\pi_{\theta}(E)\times \pi_{\theta'}(E)$. If $E$ was literally a Cartesian product, we would be able to apply Theorem \ref{thm.ring} to obtain expansion. This does not quite work as $E$ is only dense in a Cartesian product, and we seek an estimate that holds simultaneously for all ``dense'' subsets $G$ of $E$, a fact that is crucial in most if not all applications of the theorem. This is achieved through Balog-Szemer\'edi-Gowers, the robustness of energy when passing to subsets, and the following simple but clever argument due to Bourgain:

\begin{lemma} \label{lem.proj-product-case}
    Fix $x\in \R$ and let $\pi(a,b)=a+xb$. Let $A,B\subset \R$ be bounded sets. Then, for any $G\subset A\times B$ we have
    \[
        |\pi(G)| \ge \frac{|G||A+x A|}{|A-A||A-B|}.
    \]
\end{lemma}
\begin{proof}
    We start by noting that
\begin{equation}
    1_{A \times A} \leq \frac{1}{|G|}\int_{(A-A) \times (B-A)}1_{G-y}\,dy. 
\end{equation}
Indeed, for $(a,a') \in A \times A$, for any $(b_1, b_2)\in G$, we have $y=(b_1-a, b_2-a')\in (A-A)\times(B-A)$ and $(b_1, b_2)- y = (a, a')$. It follows that
\begin{equation} \label{eq.ruzsa}
    1_{A+xA} \leq \frac{1}{|G|}\int_{(A-A) \times (B-A)}1_{\pi(G) - \pi(y)}dy.
\end{equation}
Integrating both sides of \eqref{eq.ruzsa}, we obtain the desired inequality.
\end{proof}
In our application below, $A$ will satisfy the assumptions of Theorem \ref{thm.ring} for a suitable measure $\mu$, so that $|A+xA|$ is large for some $x\in\spt\mu$. Assuming $|A|$, $|B|$ and $|A-B|$ are roughly comparable, the above lemma then yields expansion for $|\pi(G)|$ simultaneously for all dense subsets $G$ of $A\times B$.

The following proposition is the main step in the proof of Theorem \ref{thm.projhaus}. It has nearly the same assumptions, and a slightly weaker conclusion: the dense set is found inside a set $F$ of ``large'' measure inside $E$, rather than inside $E$ itself.
\begin{proposition}\label{prop.proj1}
     Let $0 < \alpha <2, \beta, \kappa > 0$. There exists a universal $C\ge 1$, and $\eta > 10\e>0$, $\delta_0 \in (0,1]$ depending on $\alpha, \beta, \kappa$ only so that the following holds for all $0 < \delta < \delta_0$. 

     Let $\nu$ be a $(\kappa,\delta^{-\e})$-measure on $S^1$. Let $E\subset B^2(0,\delta^{-\e})$ be a $(\delta,\beta,\delta^{-\e})$-set of measure $\le\delta^{2-\alpha}$.  There exist $\Theta \subset \spt\nu$ with $\nu(\Theta) > 1- \delta^{\e}$, and  a union of $\delta$-balls $F \subset E$ with $|F | > \delta^{C\eta}|E|$, so that for all $\theta \in \Theta$ we have
    \begin{equation}
        |\pi_\theta(G)| > \delta^{-\eta} |E|^{1/2} \text{ for any  union of $\delta$-balls } G \subset F \text{ with } |G| > \delta^\e|F|.
    \end{equation}
\end{proposition}
\begin{proof} Fix $\eta>0$ and suppose the result is false for this choice of $\eta$, a suitably large universal constant $C$, and $\e<\eta/10$ sufficiently small. We will derive a contradiction if $\eta$ is sufficiently small in terms of $\alpha, \beta, \kappa$. Throughout the proof we assume that $\delta\ll_{\alpha,\beta,\kappa,\eta,\e} 1$ so that various implicit constants are absorbed into small powers of $\delta^{-1}$.

By making $\kappa$ smaller if necessary, we may assume that $\kappa<\beta$. By Lemma \ref{lem.energy}, we have $I_{(\kappa+\beta)/2}(\mu_E) \lesssim \delta^{-\e}$. By Theorem \ref{thm.kauf} and Markov's inequality, we have that 
\[
    \nu(\Theta_{\text{bad}})\lesssim \delta^{2\e}, \quad\text{where } \Theta_{\text{bad}} = \bigl\{\theta\in S^1: I_{\kappa}(\pi_{\theta}\mu_E) > \delta^{-3\e}\bigr\}.
\]
    
\textbf{Claim}. There are $\theta_1, \theta_2, \theta_3\in S^1$, and $E_3\subset E_2\subset E_1\subset E$ such that 
	\begin{enumerate}[(i)]
		\item \label{it:i} $|E_{i+1}|\ge \delta^{\e} |E_i|$ for $i=0,1,2$, where we set $E_0=E$, 
		\item  \label{it:ii} $|\theta_i-\theta_j|\ge \delta^{10\e/\kappa}$ for $i\neq j$, 
		\item \label{it:iii} $|\pi_{\theta_i}(E_i)| \leq \delta^{-\eta}|E|^{1/2}$ for $i=1, 2, 3$,
		\item \label{it:iv} $I_\kappa(\pi_{\theta_i}\mu_E) \leq \delta^{-3\e}$ for $i=1,2,3$.
	\end{enumerate}
	
To prove the claim, we will show that for any $E'\subset E$, $|E'|\geq \delta^{3\e}|E|$ and any $\Omega$ of at most $2$ unit vectors, there exist $\theta' \in \spt \nu\setminus \Theta_{\text{bad}}$ and $E''\subset E', |E''|\ge \delta^{\e}|E'|$ such that 
	\begin{enumerate}[(a)]
		\item \label{it.a} $|\theta'-\theta| \geq  \delta^{10\e/\kappa}$ for any $\theta\in \Omega$, 
		\item \label{it.b} $|\pi_{\theta'}(E'')| \leq \delta^{-\eta}|E|^{1/2}$. 
	\end{enumerate}
	Once the above statement is proved, we can apply it iteratively three times, starting with $E'=E$. Each step produces $E_{i+1}\subset E_i$ with $|E_{i+1}|\ge \delta^{\e}|E_i|$ and $\theta_i$ separated from the previous $\theta_j, j<i$.

    Since $\nu$ is a $(\kappa, \delta^{-\e})$-measure,
    any ball of radius $\delta^{10\e/\kappa}$ has $\nu$-measure at most $\delta^{9\e}$. This implies that 
    \[
        \nu\bigl(\Omega^{(10\e/\kappa)}\cup \Theta_{\text{bad}}\bigr)\lesssim \delta^{2\e},
    \]
    Since we are assuming that the claimed result is false (in particular, the choice $F=E'$ does not work), there exist $\theta'\notin\Theta_{\text{bad}}$ and a union of $\delta$-balls $E''\subset E'$ with $|E'|\geq \delta^{\e}|E|$ such that \eqref{it.a}--\eqref{it.b} hold. The claim then follows.

	By a projective transformation, we may suppose that 
	\begin{equation}\label{eq.theta}
		\arg\theta_1=0, \arg\theta_2=\pi/2, \arg\theta_3=\pi/4
	\end{equation}
	up to replacing $\e$ by $O(\e/\kappa)$. This step will be explained at the conclusion of the proof. 

    By \eqref{it:i} in the claim, the set $E_3$ satisfies $|E_3|\geq \delta^{3\e}|E|$. In turn, from \eqref{it:iv} in the Claim and Lemma \ref{lem.energy}, for $i=1,2$ we have
        \[
            I_{\kappa}(\pi_{\theta_i}\mu_{E_3}) \le \delta^{-6\e} I_{\kappa}(\pi_{\theta_i}\mu_{E}) \le \delta^{-\eta}.
        \]
        We apply Lemma \ref{lem.energy-to-frostman} to $\pi_{\theta_i}\mu_{E_3}$ with $\kappa$ in place of $s$, $\delta^{-\eta}$ in place of $K$, and $O_{\kappa}(1)$ in place of $L$. We conclude that for $i=1,2$ there are $(\delta,\kappa/2,\delta^{-2\eta})$-sets $A_i\subset \pi_{\theta_i}(E_3)$ such that
        \[
            \pi_{\theta_i}\mu_{E_3}(A_i) =  \mu_{E_3}(E_3 \cap \pi_{\theta_i}^{-1}A_i)\ge 3/4.
        \]
        Let 
            \[ 
                E' = E_3 \cap \pi_{\theta_1}^{-1}A_1 \cap \pi_{\theta_2}^{-1}A_2. 
            \] 
            Then $|E'|\gtrsim |E_3| \geq \delta^{3\e}|E|$ and, by the assumption \eqref{eq.theta}, we have 
            \[ 
                E'\subset A_1\times A_2.
            \]   
            By part \eqref{it:iii} of  the Claim, 
            \[   
                |A_i| \leq |\pi_{\theta_i}(E_3)| \leq \delta^{-\eta}|E|^{1/2}, \quad i=1,2.
            \]
             Hence, since $\e<\eta/10$,
        \begin{align}
            |E'| &\geq \delta^{\eta}|A_1||A_2|, \label{eq.E3-dense}\\
        |A_i|&\geq \frac{|E'|}{|A_{2-i}|} \geq \delta^{\eta}|E'|^{1/2}\ge \delta^{2\eta}|E|^{1/2} \quad i=1,2. \label{eq.E3-almost-prod}
        \end{align}

        We would like to apply Lemma \ref{lem.proj-product-case} to $A_1, A_2$ and dense subsets of $G$ of $E'$. The issue is that we do not have good control on $|A_1 - A_2|$. We will get around this by using the third projection $\pi_{\theta_3}$ and applying Balog-Szemer\'edi-Gowers.

        From now on, we use the notation $X\lessapprox Y$ to mean that $X\lesssim \delta^{-O(\eta)} Y$, and likewise for $\gtrapprox$. Using \eqref{eq.E3-almost-prod} and \eqref{it:iii} again, this time for $\theta_3$, we have
        \[
           |A_1\overset{E'}{+}A_2|_{\delta} = |\pi_{\theta_3}(E')|_{\delta} \leq \delta^{-\eta}|E|^{1/2} \le \delta^{-2\eta} |A_1|_{\delta}^{1/2} |A_2|_{\delta}^{1/2}.
        \] 
        In light of \eqref{eq.E3-dense}, we are able to apply Balog-Szemer\'edi-Gowers (Proposition \ref{prop.bsg}) with $K=\delta^{-2\eta}$, to obtain sets  $B_i\subset A_i$, which we can take to be union of $\delta$-balls, such that:
\begin{equation}\label{eq.bsg1}
	 |B_i| \gtrapprox |A_i| \overset{\eqref{eq.E3-almost-prod}}{\gtrapprox}  |E|^{1/2}, \quad i=1,2,
	\end{equation}  
	\begin{equation} \label{eq.bsg2} 
	 |B_1+B_2|_{\delta}\lessapprox |A_1|_{\delta}^{1/2} |A_2|_{\delta}^{1/2},
		\end{equation}
		and 
		\begin{equation}\label{eq.largesubset}
			\left|E'\cap (B_1\times B_2)\right|_{\delta} \gtrapprox |A_1|_{\delta} |A_2|_{\delta} \overset{\eqref{eq.E3-almost-prod}}{\gtrapprox}  |E|.
		\end{equation}
        Recall that $A_i$ is a $(\delta, \kappa/2, \delta^{-2\eta})$-set. It follows from \eqref{eq.bsg1} that  $B_i$ is a  $(\delta, \kappa/2, \delta^{-O(\eta)})$-set.  
		
		Let \begin{equation}\label{eq.up1}
			F= E' \cap (B_1 \times B_2),
		\end{equation}
        which by \eqref{eq.largesubset} satisfies 
		\begin{equation}\label{eq.up2}
			|F| \gtrapprox |E|.
		\end{equation}
        Enlarging $F$ slightly, we may assume that $F$ is a union of $\delta$-balls. We have reached the endgame of the proof. Recall that we are still assuming that the conclusion of the proposition fails, in particular for this choice of $F$. This means that there exists a set $\Theta_{\text{bad}}'\subset S^1$, which we can take to be a union of $\delta$-balls, with $\nu(\Theta_{\text{bad}}')\geq \delta^{\e}$ such that for any $\theta \in \Theta_{\text{bad}}'$, there exists some union of $\delta$-balls $G\subset F$ such that $|G|> \delta^{\e} |F|$, with
        \begin{align}
            |G| &> \delta^{\e} |F| \overset{\eqref{eq.up2}}{\gtrapprox} |E|, \label{eq.projfail1}\\
	        |\pi_{\theta} (G)|  &\leq \delta^{-\eta} |E|^{1/2}. \label{eq.projfail2} 
        \end{align} 
	
 Let $\nu'$ be the image of $\nu_{\Theta'_{\text{bad}}}$ under the map $\theta \mapsto \tfrac{1}{2\pi}\arg \theta  \in [0,1]$. Since $\nu$ is $(\kappa,\delta^{-\e})$ and $\nu(\Theta'_{\text{bad}})\ge \delta^{\e}$, this measure will be $(\ka/2,\delta^{-3\e})$, owing to the defined map being bi-Lipschitz. 
 
 We are now in a position to apply Theorem \ref{thm.ring} to $B_2$ and $\nu'$. The conclusion is that there is $c = c(\kappa,\sigma) > 0$ so that, assuming $\e,\eta\ll_{\kappa,\sigma} 1$, there exists $x \in \spt \nu'$ so that
\begin{equation}\label{eq.thm2concl}
    |B_1 +x B_1| > \delta^{-c}|B_1| \overset{\eqref{eq.bsg1}}{\gtrapprox} \delta^{-c}|E|^{1/2}.
\end{equation}
On the other hand, it follows from \eqref{eq.bsg1}, \eqref{eq.bsg2} and Corollaries \ref{cor.simple} and \ref{cor.sum-to-difference} that 
\begin{equation} \label{eq:Bi-small-diff}
  |B_1-B_i| \lessapprox  |E|^{1/2}, \quad i=1,2.
\end{equation}
Applying Lemma \ref{lem.proj-product-case} with $A=B_1$, $B=B_2$, and recalling \eqref{eq.projfail1}, \eqref{eq.projfail2}, \eqref{eq.thm2concl}, \eqref{eq:Bi-small-diff}, we conclude that
\[
   |E|^{1/2} |\pi_{\theta}(G)| \gtrapprox \frac{|G||B_1 + x B_1|}{|B_1 - B_1||B_1 - B_2|} \gtrapprox \delta^{-c}|E|^{1/2}.
\]
Recalling that $\gtrapprox$ hides $\delta^{O(\eta)}$ factors, we have reached a contradiction if $\eta$ is sufficiently small in terms of $c$, finishing the proof of the proposition under the assumption \eqref{eq.theta}.

 Finally, we  apply a projective transformation to reduce to the case of \eqref{eq.theta}. 
 Let $f: \R^2 \rightarrow \R^2$ be the projective transformation that maps the lines 
 \[
 y = \theta_jx, \ j = 1,2,3
 \]
to the $x$-axis, $y$-axis, and the line $y=x$. By \eqref{it:ii}, the map $f|_E:E\to f(E)=:E_0$ is bi-Lipschitz  with constant $\lesssim \delta^{-10\e/\kappa}$ - here we use that $E\subset B^2(0,\delta^{-\e})$.

Then $|E_0|\gtrsim \delta^{10\e/\kappa } |E|$, and it is readily checked that $E_0$ is a $(\delta, \beta, \delta^{-O(\e/\kappa)})$-set contained in $B^2(0, \delta^{-O(\e/\kappa)})$,  and that $\mu_0:=f\mu$ is a $(\kappa, \delta^{-O(\e/\kappa)})$-measure. 

Slightly abusing notation, let $f:S^1\to S^1$ also denote the induced map on directions, which is also bi-Lipschitz with constant $\lesssim \delta^{-10\e/\kappa}$. An application of the case \eqref{eq.theta} with $E$ replaced by $E_0$ and $\mu$ replaced by $\mu_0$ will yield a set $\Theta$ with $\mu_0(f(\Theta)) > 1-\delta^{\e}$ such that for any $\theta\in \Theta$ and any union of $\delta$-balls  $G_0\subset E_0$ with 
\begin{equation}\label{eq.conclusion}
	|\pi_{f(\theta)} (G_0)|> \delta^{-\eta_0}|E_0|^{1/2}
\end{equation}
for some $\eta_0 > 0$ provided that $\e, \delta$ are small enough. By reducing $\eta_0$ slightly by $O(\e/\kappa)$, and applying $f^{-1}$ to \eqref{eq.conclusion} gives the desired bound provided that $\eta$ and in turn $\e$ are small enough.
\end{proof}

We can now upgrade Proposition \ref{prop.proj1} to Theorem \ref{thm.projhaus} via an exhaustion and averaging argument that has appeared in several previous works (see e.g. \cite{He20}).

    \begin{proof}[Proof of Theorem \ref{thm.projhaus}]
            Let $0  < \e <\tfrac{\eta}{100}$ and $\delta > 0$ all be small enough to ensure the validity of multiple applications of Proposition \ref{prop.proj1} with $\e/3$ in place of $\e$.

            Let $F_1 \subset E$, $\Theta_1 \subset \spt\mu$ be such that the conclusion of Proposition \ref{prop.proj1} holds. If $ |E \setminus  F_1| \leq \delta^{2\e}|E|$, then end the process. Note that
            \begin{equation}
                |E \setminus F_1 \cap B(x,r)| \leq \delta^{-\e}r^\beta|E| < \delta^{-3\e}r^\beta |E \setminus F_1|.
            \end{equation}
        Thus, we can apply Proposition \ref{prop.proj1} to $E \setminus F_1$ and $\mu$, to obtain $F_2 \subset E \setminus F_1$, $\Theta_2 \subset \spt\mu$ satisfying the conclusion of the proposition. If $|E \setminus (F_1 \cup F_2)| \leq \delta^{2\e}|E|$ then end the process. Otherwise, we apply Proposition \ref{prop.proj1} to $E \setminus (F_1 \cup F_2)$ to obtain $F_3 \subset E \setminus (F_1 \cup F_2)$, $\Theta_3 \subset \spt\mu$. We continue inductively in this manner. The process must end since $|E|$ is finite and $|F_j| \geq \delta^{2\e+C\eta}$.

        At the conclusion of this process, we have found a collection $\{F_1,\dots,F_N\}$ of disjoint subsets of $E$, and another collection $\{\Theta_1,\dots,\Theta_N\}$ of subsets of $\spt \mu$, such that for all $1 \leq j \leq N$:
            \begin{enumerate}[(i)]
                \item  \label{it.Thetaj} $\mu(\Theta_j) > 1-\delta^{\e/3},$
                \item  \label{it.E} $\Big|E\setminus \bigcup_{i=1}^N F_j\Big| \leq \delta^{2\e}|E|,$
                \item  \label{it.conclusion} If $\theta \in \Theta_j,$ then 
                \begin{equation}
            |\pi_\theta(G)| > \delta^{-\eta}|E|^{1/2} \text{ for all unions of $\delta$-balls } G \subset F_j \text{ with } |G| > \delta^{3\e}|F_j|.
                \end{equation}
            \end{enumerate}
            
            Consider the product measure on $\mu \times \rho$, where $\rho$ is the (non probability) measure on $\{ 1, \dots, N \}$, where $j$ has weight $|F_j|$. Let $\Theta$ be the set of $\theta\in \spt \mu$ such that 
            \begin{equation}
       	\left|\bigcup_{j: \theta\in \Theta_j} F_j\right| =\sum_{j: \theta\in \Theta_j} |F_j|  = \rho\{ j: \theta \in \Theta_j \} \geq  (1-\delta^{\e})|E|. 
            \end{equation}
        By \eqref{it.Thetaj} and \eqref{it.E}, if $\delta$ is small enough,
        \[
            \int \rho\{ j: \theta \in \Theta_j \} d\mu(\theta) = (\mu\times\rho)(\{ (\theta, j): \theta\in \Theta_j\}) \ge (1-\delta^{\e/2})|E|.
        \]    
        By Markov's inequality, this implies that
        \[
        \mu(\Theta) \geq 1-\delta^{\e/2}.
        \]  
        Finally, if $\theta\in \Theta$ and $G\subset E$ satisfies $|G|\geq \delta^{\e}|E|$, then there is $j$ such that $\theta\in \Theta_j$ and $|G\cap F_j|\geq \delta^{2\e}|F_j|$, which in light of \eqref{it.conclusion} yields the  claim with $\e/2$ in place of $\e$.
    \end{proof}

%\bibliographystyle{plain}
%\bibliography{references.bib}

\end{document}